\documentclass[12pt,leqno,a4paper]{elsarticle}
\usepackage{amssymb,amsmath}
\usepackage[encapsulated]{CJK}

\usepackage{geometry} 
\usepackage{amsthm}
\usepackage{array} 
\usepackage{paralist} 
\usepackage{verbatim} 
\usepackage{amsmath}
\usepackage[all]{xy}

\newtheorem{theorem}{Theorem}[section]

\theoremstyle{definition}
\newtheorem{definition}{Definition}[section]
\newtheorem{remark}{Remark}[section]
\newtheorem{example}{Example}[section]

\title{$A$-coupled-expanding and distributional chaos}

\author[pyon]{Cholsan Kim}
\author[pyon]{Hyonhui Ju}
\author[vien]{Peter Raith}
\ead{Peter.Raith@univie.ac.at}
\address[pyon]{Department of Mathematics,
\textbf{Kim Il Sung} University, Pyongyang, D.P.R.Korea}
\address[vien]{Fakult\"at f\"ur Mathematik,
Universit\"at Wien, Nordbergstra{\ss}e~15,
1090 Wien, Austria}

\begin{document}

\begin{abstract}
The concept of $A$-coupled-expanding map, which is one of the more natural and useful ideas generalized the horseshoe map, is well known as a criterion of chaos. It is well known that distributional chaos is one of the concepts which reflect strong chaotic behaviour. In this paper, we focus the relations between $A$-coupled-expanding and distributional chaos. We prove two theorems that give sufficient conditions for a strictly $A$-coupled-expanding map to be distributionally chaotic in the
senses of two kinds, where $A$ is an $m\times m$ irreducible transition matrix.
\end{abstract}


\maketitle

\section{Introduction}

The concept of $A$-coupled-expanding map  which  has been recognized as one of the important criteria of chaos has been defined in \cite{shijuchen}. But when it comes to the $A$-coupled-expanding map, it has to be noted that this notion goes back to the notion of turbulence introduced by Block and Coppel in \cite{blockcop}, which has been considered as an important property of chaotic behaviour for one-dimensional dynamical system. A continuous map
$f:I\rightarrow I$, where $I$ is the unit interval, is said to be \emph{turbulent} if there exist closed nondegenerate subintervals $J$ and $K$ with pairwise disjoint interiors such that $f(J)\supset J\cup K, f(K)\supset J\cup K$. Furthermore, it is said to be \emph{strictly turbulent} if the subintervals $J$ and $K$ can be chosen disjoint.
Actually, essentially the same concept was studied in one-dimensional dynamical system before by Misiurewicz in \cite{misirewhsh} and \cite{misirehsbr}. He called this property ``horseshoe'', because it is similar to the Smale's horseshoe effect. Let $f:I\rightarrow I$ be an interval map and $J_1, {\dots}, J_n$ be nondegenerate subintervals with pairwise disjoint interiors such that
$J_1\cup\dots\cup J_n\subset f(J_i)$ for $i=1, {\dots}, n$.
Then $(J_1, {\dots}, J_n)$ is called an \emph{$n$-horseshoe}, or simply a \emph{horseshoe} if $n\geq2$.

This concept has been more generalized for general metric spaces.
In \cite{shiyuturb} it has been extended from the concept of turbulence for continuous
interval maps to maps in general metric spaces. Since the term ``turbulence'' is
well-established in fluid mechanics,  they changed the term ``turbulence'' to
``coupled-expansion'' (\cite{shichencep}, \cite{shiyusbrep}). There were some results on
chaos for coupled-expanding maps (\cite{blockcop}, \cite{shichenccms}, \cite{shichencep},
\cite{shiyuturb}, \cite{shiyusbrep} and \cite{zhangshis}). Later, this notion was further extended to coupled-expanding maps for a transitive matrix $A$(simply called  $A$-coupled-expanding map) in \cite{shijuchen}, which  is the same as the concept of coupled-expanding map if each entry of
the matrix $A$ equals to 1.

Recently, by applying symbolic dynamical system theory, many important results about criteria of chaos using the $A$-coupled-expanding map have been established. For instance, it has been verified that under certain conditions, the strictly $A$-coupled-expanding map is chaotic in the sense of Li-Yorke or Devaney or Wiggins(\cite{kulopro}, \cite{shichenccms}, \cite{shichencep}, \cite{shijuchen}, \cite{shizhangcem}). Also, it was proved that some $A$-coupled-expanding map has positive topological entropy in \cite{rijuentacem}, \cite{zhangshichencem}.

On the other hand, since the concept of distributional chaos have been introduced for the first
time in \cite{schweizsmitdc}, a large number of papers  have been devoted to the study on
distributional chaos, including researches on the relations between the distributional
chaos and many other definitions of chaos such as Li-Yorke chaos, Devaney chaos and so on(see for example \cite{balsmistef}, \cite{fortivnc}, \cite{oprochaddc},
\cite{oprochadcsc}, \cite{oprochadcr}, \cite{smitalvnc}, \cite{wangxiang} and
\cite{wanghuanghuan}). In \cite{kulopro} was shown that if an $A$-coupled-expanding map $f$ satisfies some expanding conditions, then there is a $f$-invariant set $\Lambda$ such that $f|_\Lambda$ is conjugate to a subshift of finite type, consequently, it is distributionally chaotic(See Remark 3.3 of \cite{kulopro}). As above mentioned, recently the concept of $A$-coupled-expanding map which is one of the more natural and useful ideas generalized of the horseshoe map, is well known  as a criterion of chaos, but there is no result yet on the relation between $A$-coupled-expanding and distributional chaos except~\cite{kulopro}. Furthermore, distributional chaos is recognized as one of the concepts reflected strong chaotic behavior. This implies that it is natural to study the relation between $A$-coupled-expanding and distributional chaos or complement conditions for one to be other.

The rest of this paper is organized as follows. In Section~2 some basic concepts and notations which will be used later are introduced. In Section~3 a sufficient condition for a strictly $A$-coupled-expanding map on a compact metric space to be distributionally chaotic in a sequence is established (Theorem~\ref{thm-dcseq}) and give an example which illustrates that our results is remarkable. Finally it is proved that under stronger conditions the map is distributionally chaotic (Theorem~\ref{thm-perdcgen}).

\section{Preliminaries}

Let $\mathbb{N}$ be the set of all  natural numbers and $\mathbb{N}_0=\mathbb{N} \cup\{0\}$. For $m \ge 2$, let $A=(a_{ij})$ be an $m\times m$ matrix. The definitions of transition and irreducible matrix~$A$ follow \cite{shizhangcem}.

Set $\mathcal{A}=\{1,\dots , m\}$ and the element of the $\mathcal{A}$ is called~\emph{alphabet}. It is well known that the set $\Sigma _m=\{\alpha =(a_0 a_1 \ldots)~|~a_i \in \mathcal{A}, i\in\mathbb{N}_0\}$ is a compact metric space with the metric
\begin{displaymath}
 \rho (\alpha, \beta)=\left\{ \begin{array}{ll}
0 & \textrm{if $\alpha =\beta$,}\\
2^{-(k+1)} & \textrm {if $ \alpha \ne \beta$ and $k=\min \{i|a_i\ne b_i\}$,}
\end{array} \right.
\end{displaymath}
where $\alpha =(a_0a_1 \ldots ), \beta=(b_0 b_1 \ldots ) \in \Sigma _m$. For $\alpha =(a_0a_1 \ldots a_i\ldots )$, the subscript $i\in\mathbb{N}_0$ is called the number of $a_i$ in the $\alpha$.
Define the shift map $\sigma:\Sigma_m\rightarrow\Sigma_m$ by $\sigma(\alpha)=(a_1a_2\ldots)$, where $\alpha=(a_0a_1a_2\ldots)\in\Sigma_m$. It is well known that for an $m\times m$ transition matrix $A=(a_{ij})$, $\Sigma_A=\{\beta =(b_0b_1\ldots) \in
\Sigma _m~|~\; a_{b_ib_{i+1}}=1, i\in\mathbb{N}_0 \}$ is a compact subset of $\Sigma _m$, which is invariant under
the shift map $\sigma$. The map $\sigma_A=\sigma|_{\Sigma _A}$ is said to be the subshift of finite type for matrix $A$.

For $(a_0a_1 \ldots) \in \Sigma _A$ and $n\in\mathbb{N}$, the type of $(a_ia_{i+1}\ldots a_{i+n})$$(i\in\mathbb{N}_0)$ is called \emph{admissible word} for matrix $A$ and denote $|(a_ia_{i+1}\ldots a_{i+n})|=n+1$ and call it the \emph{length} of the $(a_ia_{i+1}\ldots a_{i+n})$.
For $s, t\in\mathbb{N}$, let $u=(u_0u_1\ldots u_s)$, $v=(v_0v_1\ldots v_t)$, where $u_i, v_i\in\mathcal{A}$. Denote $uv=(u_0u_1\ldots u_sv_0v_1\ldots v_t)$ and call it the \emph{combination} of $u$ and $v$. Also, for $n\in\mathbb{N}$, denote $u^n=uu\ldots u$($n$ times repeated).
The following definition can be found in \cite{shijuchen}, \cite{oprochaddc}  and  \cite{liuwangchudcdc}.

\begin{definition}(\cite{shijuchen})
Let $(X,d)$ be a metric space and $f:D\subset X \rightarrow X$.
Suppose that $A=(a_{ij})$ is an $m\times m$ transition matrix for some~$m\ge 2$. If there exist $m$ nonempty subsets $\Lambda_i$ ($1\leq i\leq m$) of~$D$ with pairwise disjoint interiors such that
\begin{displaymath}
f(\Lambda_i)\supset\bigcup\limits_{\substack{j\\a_{ij}=1}}\Lambda_j
\end{displaymath}
for all~$1\leq i\leq m$, then $f$ is said to be an \emph{$A$-coupled-expanding map}
in $\Lambda_i$ ($1\leq i\leq m$). Moreover, the map~$f$ is said to
be a \emph{strictly $A$-coupled-expanding map} in $\Lambda_i$ ($1\leq i\leq m$) if
$d(\Lambda_i, \Lambda_j)>0$ for all~$1\leq i\neq j\leq m$.
\end{definition}

\begin{definition}(\cite{oprochaddc})
Let $(X,d)$ be a compact metric space and $f:X \rightarrow X$ be a continuous map. For $x,y\in X$ set
\begin{displaymath}
F_{xy}^{(n)}(t)=\frac {1}{n}  \# \{ i:d(f^i (x), f^i (y)) < t, 0\le i <n \}\,,
\end{displaymath}
where $\#A$ denotes the cardinal number of set the~$A$. Furthermore set
\begin{displaymath}
F_{xy}(t)=\liminf\limits_{n\rightarrow\infty} F_{xy}^{(n)}(t)\quad\text{and}\quad
F_{xy}^*(t)=\limsup\limits_{n\rightarrow\infty} F_{xy}^{(n)}(t)\,.
\end{displaymath}
\end{definition}

The functions~$F_{xy}(t)$ and $F_{xy}^*(t)$ are called \emph{lower} or \emph{upper distribution functions} of $x$ and $y$ respectively. We consider the following three conditions:
\begin{displaymath}
 \begin{array}{ll}
\text{D(1)}  \quad F_{xy}^*\equiv 1 \quad\text{and}\quad\exists s>0,  F_{xy}(s)=0,\\
\text{D(2)}  \quad F_{xy}^*\equiv 1 \quad\text{and}\quad F_{xy}<F_{xy}^*,\\
\text{D(3)}  \quad F_{xy}<F_{xy}^*.
\end{array}
\end{displaymath}
If a pair of  points~$(x, y)$ satisfies the condition D($k$) ($k=1,2,3$), then the pair~$(x,y)$ is called \emph{distributionally chaotic pair of type~$k$}. A subset of $X$ containing at least two points is called \emph{distributional scrambled set of type~$k$} if any two points of the subset form distributionally chaotic pair of type~$k$. Finally $f$ is said to be \emph{distributionally chaotic of type~$k$} if $f$ has an uncountable distributional scrambled set of type~$k$.

\begin{definition}(\cite{liuwangchudcdc})
Let $(X, d)$ be a compact metric space and $f: X \rightarrow X$ be a continuous map. Suppose that $(p_i)$ is an increasing sequence of positive integers, $x, y \in X$ and $t>0$. Set
\begin{displaymath}
\begin{array}{ll}
F_{xy}(t, (p_i))=\liminf \limits_{n \rightarrow \infty }
\frac {1}{n} \# \{i:d(f^{p_k}(x), f^{p_k}(y))<t, 1\le k \le n \},\\
F_{xy}^*(t, (p_i))=\limsup \limits_{n \rightarrow \infty }
\frac {1}{n} \# \{i:d(f^{p_k}(x), f^{p_k}(y))<t, 1\le k \le n \}.
\end{array}
\end{displaymath}
\end{definition}

A subset~$D \subset X$ is called \emph{distributionally chaotic set in a sequence} (or \emph{in the sequence $(p_i)$}) if for any $x,y\in D$ with $x\ne y$
\begin{enumerate}[(1)]
\item there is an $\varepsilon >0$ with $F_{xy}(\varepsilon, (p_k))=0$, and
\item $F_{xy}^*(t, (p_k))=1$ for every $t>0$.
\end{enumerate}
If $f$ has an uncountable distributionally chaotic set in a sequence, the map~$f$ is said to be \emph{distributionally chaotic in a sequence}. Obviously distributional chaos of type 1 implies distributional chaos in a sequence.

\section{Main results}

Assume that $(X, d)$ is a compact metric space, $m \ge 2$ and for $1\le i \le m$ the
sets~$V_i$ are compact subsets of~$X$ with pairwise disjoint interiors. Let
$f:\bigcup \limits _{i=1}^m V_i\rightarrow X$ be a continuous map and
$A=(a_{ij})$ be an irreducible $m \times m$ transition matrix satisfying
that
\begin{equation*}
\text{there exists an~$i_0$ with }\sum_{j=1}^m a_{i_0j}\ge 2\,.
\tag{\lower 0.6ex\hbox{*}}
\end{equation*}

Suppose that $f$ is a strictly $A$-coupled-expanding map in the $V_i$ ($1\le i \le m$). For any
$\beta =(b_0b_1 \ldots) \in \Sigma _A$, set $V_\beta =\bigcap \limits_{n=0}^\infty
f^{-n}(V_{b_n})$,where $f^0$ is the identity map, and for any admissible word for the matrix $A$, $c=(c_0c_1 \ldots c_n)$, put
$V_c=V_{c_0c_1 \ldots c_n}=\bigcap \limits_{i=0}^n f^{-i}(V_{c_i})$. Then $V_c$ is a nonempty compact subset. Moreover, $V_{c_0c_1 \ldots c_{n-1}}\supset
V_{c_0c_1 \ldots c_{n-1}c_n}$ and $f(V_{c_0c_1\ldots c_n})=V_{c_1 \cdots c_n}$. Therefore $V_\beta=\bigcap \limits_{n=0}^\infty
V_{b_0b_1 \ldots b_n}$ and hence $V_\beta$ is nonempty and compact. If
$(c_0c_1 \ldots c_n), (d_0d_1 \ldots d_n)$ are two different admissible words for the matrix $A$, then
$V_{c_0c_1 \ldots c_n}\cap V_{d_0d_1\ldots d_n}=\emptyset$ (see \cite{zhangshichencem}).


\begin{theorem}\label{thm-dcseq}
Let $(X, d)$ be a compact metric space, $m \ge 2$. Suppose that $f:X\rightarrow X$ is a continuous map and $A$ is an irreducible $m \times m$ transition matrix satisfying the assumption~\textup{(\lower 0.6ex\hbox{*})}. Let $f$ be strictly $A$-coupled-expanding in compact sets $V_i(1\le i \le m)$. If there exists
an~$\alpha =(a_0a_1 \ldots) \in \Sigma_A$ such that
$V_{\alpha}=\bigcap _{n=0}^\infty f^{-n}(V_{a_n})$ is a singleton,
then there exists a sequence~$(p_k)$ such that $f$ is distributionally chaotic in the sequence~$(p_k)$.
\end{theorem}

\begin{proof}
From the assumptions and the facts presented earlier in this section, it follows that for any $n\in \mathbb{N}_0$,
\begin{equation}
f^n(V_\alpha)=V_{\sigma^n(\alpha)}
\end{equation}
and $f^n(V_\alpha)$ also is singleton. Obviously there is at least one alphabet appearing infinitely in $\alpha =(a_0a_1 \ldots)$ and without loss of generality we may assume that $a_0$ appears infinitely in~$\alpha$. This means that there exists a strictly increasing sequence~$(\nu_k)$ in~$\mathbb{N}$ with $a_{\nu_k}=a_0$ for all~$k\in {\mathbb N}$. Now for $k\in {\mathbb N}$, set $u_k=(a_0a_1 \ldots a_{\nu_k-1})$. Obviously, for $n\in\mathbb{N}$, $u_k^n$ is an admissible word for the matrix $A$.

Since the definition of irreducible matrix and the assumption~($*$), Lemma~2.4 of~\cite{shizhangcem}, there is an~$a' \in \cal A$ such that $a_{a'a_0}=1$, moreover, there are two different admissible words for the matrix $A$, $v_1=(a_0 \ldots a')$ and $v_2=(a_0 \ldots a')$, such that $|v_1|=|v_2|$. Obviously any combination of $u_k$, $v_1$, $v_2$ is admissible word for matrix $A$.

Define a map $\varphi:\Sigma_2\rightarrow \Sigma_A$ by
\begin{equation}\label{frm-defphi}
\varphi(c)=v_{c_0}^{s_1}u_1^{s_2}v_{c_1}^{s_3}v_{c_2}^{s_4}u_1^{s_5}u_2^{s_6}
v_{c_3}^{s_7}v_{c_4}^{s_8}v_{c_5}^{s_9}u_1^{s_{10}}u_2^{s_{11}}u_3^{s_{12}}
v_{c_6}^{s_{13}}\ldots
\end{equation}
for $c=(c_0c_1c_2c_3c_4 \ldots) \in \Sigma_2$,
where the $s_i$ are defined as following;
\begin{displaymath}
 \begin{array}{ll}
\ s_1=1, \\
\ s_2=2^1|v_{c_0}^{s_1}|,\\
\ s_3=2^2|v_{c_0}^{s_1}u_1^{s_2}|,\\
\ s_4=2^3|v_{c_0}^{s_1}u_1^{s_2}v_{c_1}^{s_3}|,\\
\ s_5=2^4|v_{c_0}^{s_1}u_1^{s_2}v_{c_1}^{s_3}v_{c_2}^{s_4}|,\\
\ s_6=2^5|v_{c_0}^{s_1}u_1^{s_2}v_{c_1}^{s_3}v_{c_2}^{s_4}u_1^{s_5}|,\\
\ldots    \ldots    \ldots \\
\end{array}
\end{displaymath}

This map~$\varphi$ is well defined and obviously bijective by the definition. By the construction of the map $\varphi$, any two elements in the $\varphi(\Sigma_2)$ coincide in the parts appearing combinations of $u_i$. Also, for any two elements of $\varphi(\Sigma_2)$, the first alphabets $a_0$ of admissible words $u_i$ or $v_j$ appear in the same places. Now, for any $\alpha\in \varphi(\Sigma_2)$, rearranging the numbers of the first alphabets $a_0$ of $u_i$ or $v_j$ in the $\alpha$, we make a strictly increasing sequence $(p_k)_{k=1}^\infty$ in $\mathbb{N}_0$. As $(\Sigma_2, \sigma)$ is chaotic in
the sense of Li-Yorke, it has an uncountable scrambled set which we denote by $S$. Put $D_0=\varphi (S)$. Then $D_0$ is also uncountable.
Since the facts presented earlier in this section, it follows that for any $\hat{\alpha}=(\hat{a}_0\hat{a}_1\ldots)\in D_0$, $V_{\hat{\alpha}}$ is nonempty and if $\hat{\alpha}\ne \hat {\beta}\in D_0$, then $V_{\hat{\alpha}}\cap V_{\hat{\beta}}=\emptyset$. Now, for each $\hat\alpha\in D_0$, we choose only one element from $V_{\hat\alpha}$ and fix it by denoting $x_{\hat\alpha}$. Set $G=\{x_{\hat\alpha}~|~\hat\alpha\in D_0\}$. 

We prove that this set $G$ is the uncountable distributional scrambled set in the sequence $(p_k)_{k=1}^\infty$. It is obvious that the $G$ is uncountable.
Using condition of this theorem  and the facts presented earlier in this section, one obtains that
$(\operatorname{diam}(V_{a_0a_1 \cdots a_n}))_{n=0}^\infty$ is a non increasing sequence and
$\lim_{n\to\infty}\operatorname{diam}(V_{a_0a_1 \cdots a_n})=0$, where
$\operatorname{diam}(V)$ denotes diameter of $V$. Therefore for every~$\varepsilon>0$ there
is an~$n_0 \in \mathbb{N}$ such that
\begin{equation}\label{frm-diamzero}
\operatorname{diam}(V_{a_0a_1 \cdots a_n}) <\varepsilon
\end{equation}
for all~$n \in \mathbb{N}$ with $n\ge n_0$. For any $\hat {\alpha}=(\hat{a}_0\hat{a}_1\ldots) \in D_0$ we can choose
a subsequence~$(p_{k_j})$ of $(p_k)$ such that $(p_{k_j})$-th alphabet is the first alphabet(that is, $a_0$) of a combination $u_i^{s_k}$. By (\ref{frm-defphi}), there is a $p_{k_j}$ such that  $(\hat{a}_{p_{k_j}}\hat{a}_{p_{k_j}+1}\hat{a}_{p_{k_j}+2}\ldots
\hat{a}_{p_{k_j+1}-1} )=(a_0a_1 \ldots a_{p_{k_j+1}-p_{k_j}-1}) =u_{p_{k_j+1}-p_{k_j}}$ and
\begin{equation}\label{frm-langthlar}
|u_{p_{k_j+1}-p_{k_j}}|=p_{k_j+1}-p_{k_j} >n_0.
\end{equation}
 From~(\ref{frm-defphi}) there is a subsequence $(s_{k'_j})$ of $(s_i)$ such that
$\hat{\alpha}$ contains combination of admissible words, $u_{p_{k_j+1}-p_{k_j}}^{s_{k'_j}}$.
Fix $j \in \mathbb{N}$ arbitrarily. Then, by (\ref{frm-diamzero}) and (\ref{frm-langthlar}), for any $i$ with $k_j \le i\le k_j+s_{k'_j}-1 $,
 it follows that $f^{p_i}(x_{\hat{\alpha}}),
f^{p_i}(x_{\hat{\beta}})\in V_{u_{p_{k_j+1}-p_{k_j}}}$ for any $x_{\hat{\alpha}}$, $x_{\hat{\beta}}\in G$. Since
$\operatorname{diam}(V_{u_{p_{k_j+1}-p_{k_j}}})<\varepsilon$, we have
$d(f^{p_i}(x_{\hat{\alpha}}), f^{p_i}(x_{\hat{\beta}}))<\varepsilon$. Thus, from the definition of $s_i$, we obtain

\begin{eqnarray*}
\lefteqn{\frac{\#\{i:d(f^{p_i}(x_{\hat{\alpha}}),
f^{p_i}(x_{\hat{\beta} }) )<\varepsilon, 1 \le i\le k_j+s_{k'_j}-1\} }{k_j+s_{k'_j}-1}} \\
  & &{} \ge\frac{s_{k'_j}}{k_j+s_{k'_j}-1} \\
  & &{} \ge\frac{s_{k'_j}}{p_{k_j}+s_{k'_j}-1}\\
  & &{} = \frac{s_{k'_j}}{2^{-k'_j+1}s_{k'_j}+s_{k'_j}-1}
\rightarrow 1(j \rightarrow \infty).
\end{eqnarray*}
By above expression and the definition of superior limit of sequence, we have
\begin{equation}\label{frm-udftone}
\limsup_{n\to\infty}\frac{\#\{i:d(f^{p_i}(x_{\hat{\alpha} }),
f^{p_i}(x_{\hat{\beta} }) )<\varepsilon, 1\le i\le n\} }{n}=1.
\end{equation}

Next set $d_0=d(V_{v_1},V_{v_2})>0$. Since any two different elements of the set $S\subset \Sigma _2$ is Li-Yorke pair, the set $\{i:c_i \ne d_i\}$ is infinite for $c=(c_0c_1 \ldots)\ne d=(d_0d_1 \ldots) \in S$. Therefore, without loss of generality, we may assume that for $\hat\alpha\ne\hat\beta\in D_0$, there exist a subsequence $(p_{q_j})_{j=1}^{\infty}$ of the sequence ($p_k$) and a subsequence $(s_{q'_j})_{j=1}^{\infty}$ of the sequence~$(s_i)$ such that ,
\begin{displaymath}
\sigma ^{p_{q_j}}(\hat{\alpha})=(v_1^{s_{q'_j }} \ldots)\text{ and }
\sigma ^{p_{q_j}}(\hat{\beta})=(v_2^{s_{q'_j }} \ldots)
\end{displaymath}
hold. Fix $j\in\mathbb{N}$ arbitrarily. For $i$ with $q_j \le i \le q_j+s_{q'_j}-1$, it follows that
$f^{p_i}(x_{\hat {\alpha}}) \in V_{v_1}$ and $f^{p_i}(x_{\hat {\beta}}) \in V_{v_2}$.
This means that $d(f^{p_i}(x_{\hat {\alpha}}), f^{p_i}(x_{\hat {\beta}})  \ge d_0$. Thus
\begin{eqnarray*}
\lefteqn{\frac{\#\{i:d(f^{p_i}(x_{\hat{\alpha} }),f^{p_i}(x_{\hat{\beta} }) )<d_0, 1\le i\le
q_j+s_{q'_j}-1\} }{q_j+s_{q'_j}-1}{}} \\
&&{}\le \frac {q_j-1}{q_j+s_{q'_j}-1}\\
&&{}\le\frac{2^{-q'_j+1}s_{q'_j}-1}{2^{-q'_j+1}s_{q'_j}+s_{q'_j}-1}
\rightarrow 0(j \rightarrow \infty),
\end{eqnarray*}
hence
\begin{equation}\label{frm-ldfzer}
\liminf_{n\to\infty}\frac{\#\{i:d(f^{p_i}(x_{\hat{\alpha} }),
f^{p_i}(x_{\hat{\beta} }) )<d_0, 1\le i\le n\} }{n}=0.
\end{equation}
From (\ref{frm-udftone}) and (\ref{frm-ldfzer}), the set $G$ is the uncountable distributional scrambled set in the sequence $(p_k)$, therefore, we can conclude that the map~$f$ is distributionally chaotic in the sequence~$(p_k)$.
\end{proof}

\begin{example}\label{exm-fgdcs}
Define $f : [0, 3]\rightarrow [0, 3]$ as following;
\begin{displaymath}
f(x)= \left\{ \begin{array}{ll}
1.5x+2,& \textrm{if $x\in [0, \frac{1}{3}]$}\\
2.5, & \textrm{if $x\in (\frac{1}{3}, \frac{2}{3}]$}\\
1.5x+1.5, & \textrm{if $x\in (\frac{2}{3}, 1]$}\\
3, &\textrm{if $x\in (1, 2]$}\\
-3x+9, &\textrm{if $x\in (2, 3]$}.
\end{array} \right.
\end{displaymath}
Set
\begin{displaymath}
A =
\left( \begin{array}{ccc}
0 & 1 \\
1 & 1 \\
\end{array} \right)
\end{displaymath}
and $V_1=[0,1]$, $V_2=[2,3]$.
This matrix $A$ is an irreducible transition matrix satisfying the assumption~\textup{(\lower 0.6ex\hbox{*})} and $f$ is a strictly $A$-coupled-expanding in $V_1, V_2$. Also this map $f$ satisfies the remaining conditions of the Theorem~\ref{thm-dcseq}. Therefore, the map $f$ has an uncountable distributional scrambled set in a sequence. However, $f$ does not satisfy the epsilon-delta condition of the Theorem 1 from~\cite{kulopro}, that is, set $B=\{2.5\}$, then $\operatorname{diam}(B)$=0 and $\operatorname{diam}(f^{-1}(B)\cap{V_1})=\frac{1}{3}$.
\end{example}


\begin{theorem}\label{thm-perdcgen}
Let $(X, d)$ be a compact metric space, $m \ge 2$. Suppose that $f:X\rightarrow X$ is a continuous map and $A$ is an irreducible $m \times m$ transition matrix satisfying the assumption~\textup{(\lower 0.6ex\hbox{*})}. Let $f$ be strictly $A$-coupled-expanding in the compact sets $V_i(1\le i \le m)$. If there exists a periodic point~$\alpha =(a_0a_1 \ldots) \in \Sigma_A$ such that
$V_{\alpha}=\bigcap _{n=0}^\infty f^{-n}(V_{a_n})$ is a singleton,
then $f$ is distributionally chaotic of type 1.
\end{theorem}

\begin{proof}
Denote the period of $\alpha$ by $T$, that is, $\sigma_A^T(\alpha)=\alpha$. Then
$\#\{ \sigma_A^n(\alpha)~|~n\in \mathbb{N}_0\}=T$. From the assumption of the theorem,
$f^n(V_{\alpha})=V_{\sigma _A^n(\alpha)}=
\bigcap \limits_{i=0}^\infty V_{a_na_{n+1}\ldots a_{n+i}}$ is a singleton for
any~$n \in \mathbb{N}_0$. Furthermore, for any~$n \in \{0,1,\ldots, T-1\}$ the sequence
$(\operatorname{diam}(V_{a_n \ldots a_{n+i}}))_{i=0}^\infty$ is nonincreasing and
satisfies $\lim_{i\to\infty}\operatorname{diam}(V_{a_n \ldots a_{n+i}})=0$. Set
$d_k=\max \{\operatorname{diam}(V_{a_n \ldots a_{n+k}})~|~0\le n \le T-1\}$
for~$k\in \mathbb{N}$. Then
\begin{equation}\label{frm-diamdkzer}
\lim_{k\to\infty}d_k=0.
\end{equation}
Since the definition of irreducible matrix and the assumption ($*$), there exists an~$a' \in \mathcal{A}$ such that $(A)_{a'a_0}=1$. There are two different admissible words $u=(a_0 \ldots a'), v=(a_0 \ldots a')$ for the matrix $A$ such that $|u|=|v|=l$. Also, since the matrix~$A$ is irreducible, there exist $i,j \in \mathcal{A}$ such that $(A)_{a_0i}=(A)_{ja_0}=1$ and there is an admissible word~$C=(i \ldots j)$ for the matrix~$A$.

For $p\in\mathbb{N}$ set $B_p=(a_0a_1\ldots a_{p-1})$. Define $\bar{B}_p $ as following; If $a_{p-1}=a_0$ we assume that $\bar{B}_p$ vanishes. Otherwise, there is a $q\in\{1,2,\dots ,T-1\}$ such that $a_{p-1}=a_q$. In this case set $\bar{B}_p=(a_{q+1}a_{q+2}\ldots a_T)$.
Construct a map  $\varphi:\Sigma_2\rightarrow\Sigma_A$ as following; Define a map~$\psi: \lbrace{1,2}\rbrace\rightarrow\lbrace{u,v}\rbrace$ as $\psi(1)=u$, $\psi(2)=v$. Now define the map $\varphi$ by
\begin{eqnarray}\label{frm-varphi3.2def}
\lefteqn{\varphi(a_0a_1a_2\ldots){}}\nonumber\\
&&{}=a_0C\psi(a_0)^{m_1}B_{p_2}\bar{B}_{p_2}C\psi(a_1)^{m_3}\psi(a_2)^{m_4}B_{p_5}\bar{B}_{p_5}C\psi(a_3)^{m_6}\psi(a_4)^{m_7}\\
&&{}\qquad\psi(a_5)^{m_8}B_{p_9}\bar{B}_{p_9}C\psi(a_6)^{m_{10}}\psi(a_7)^{m_{11}}
\psi(a_8)^{m_{12}}\psi(a_9)^{m_{13}}\ldots\nonumber
\end{eqnarray}
where $m_i$ is $2^i$th power of the length of the subword to the left of~$\psi(a_j)^{m_i}$ in~(\ref{frm-varphi3.2def}) and $p_i$ is $2^i$th power of the length of the subword to the left
of~$B_{p_i}$ in~(\ref{frm-varphi3.2def}). Obviously any subwords in the sequence defined in~(\ref{frm-varphi3.2def}) are the admissible words for the matrix $A$, that is, the map $\varphi$ is well defined. By the construction of (\ref{frm-varphi3.2def}), the map $\varphi$ is bijective. Since the shift $(\Sigma_2, \sigma)$ is chaotic in sense of Li-Yorke, $\sigma$ has an uncountable scrambled set in $\Sigma_2$. Denote the image of this scrambled set by $\varphi$  as $D_0$. For any $\hat{\alpha} \in D_0$, we choose an element from $V_{\hat {\alpha}}$ and fix it, denoting as $x_{\hat{\alpha}}$. Denote the set consisting of these elements as $D_x$($\subset\bigcup_{i=1}^m{V_i}$).

Now we prove that this  $D_x$ is uncountable distributional scrambled
set of type~1. It is obvious that $D_x$ is uncountable. Fix $x_{\hat \alpha}, x_{\hat \beta}\in D_x$ arbitrarily. (\ref{frm-diamdkzer}) implies that for every~$t>0$ there is a~$k \in \mathbb{N}$ with $d_k<t$. Let $\hat\alpha\ne\hat\beta\in D_0$. Obviously, for any $i$, the end alpabet of the combination $B_{p_i}\bar{B}_{p_i}$ is $a_0$. Denote the number of the $a_0$ in $\hat\alpha$ as $n_i$. By definition of $\varphi$, $n_i$ also is the number of the end $a_0$ of $B_{p_i}\bar{B}_{p_i}$ in $\hat\beta$. Consider $|\bar{B}_{p_i}|\leq T$, then we can see easily that $n_i \le p_i(1+2^{-i})+T$. The choice
of~$p_i$ gives that $|B_{p_i}|<|B_{p_j}|$ for~$i<j$, and that there is an~$i_0$ such that $|B_{p_i}|>k$ for~$i\ge i_0$.

Therefore, for~$i\ge i_0$,
\begin{eqnarray*}
\lefteqn{\frac{\# \{s:d(f^s(x_{\hat \alpha}),f^s(x_{\hat \beta}))<t,
0\le s<n_i\}}{n_i}}{}\\
&&{}\ge \frac{\# \{s:d(f^s(x_{\hat \alpha}),
f^s(x_{\hat \beta}))<d_k, 0\le s<n_i\}}{n_i}\\
&&{}\ge \frac{\# \{s:
\text{$(a_s\ldots a_{s+k})$ lies in $B_{p_i}$}\}}{n_i}\\
&&{}=\frac{p_i-k}{n_i}\\
&&{}\ge
\frac{p_i-k}{p_i(1+2^{-i})+T}\rightarrow 1(i\rightarrow\infty),
\end{eqnarray*}
hence
\begin{equation}\label{frm-ttudfone}
\limsup_{n\to\infty}\frac{\# \{s:d(f^s(x_{\hat \alpha}),f^s(x_{\hat \beta}))<t,
0\le s<n\}}{n}=1.
\end{equation}

Denote the set of all admissible words for the matrix $A$ with the length~ $l$ as $W$. Choose $d_0$ such that
\begin{displaymath}
0<d_0<\min \{d(V_c, V_d)~|~c\ne d \in W\}.
\end{displaymath}
Assume that $\hat{\alpha}=(\hat{a}_0\hat{a}_1 \ldots )\ne \hat{\beta}=
(\hat{b}_0\hat{b}_1 \ldots )\in D_0$. From the definition of $\varphi$ and the construction of~$D_0$ we may assume without loss of generality that there is an increasing sequence~$(r_j)$ in~$\mathbb{N}$ such that combinations $u^{m_{r_j}}$ and $v^{m_{r_j}}$ appear in the same place of $\hat\alpha$, $\hat\beta$, respectively. Denote by $\nu_j$ the original number of~$a'$ in the sequence~$\hat{\alpha}$(or~$\hat\beta$) which is the latest~$a'$ in~$u^{m_{r_j}}$(or~$v^{m_{r_j}}$). Since $\nu _j=m_{r_j}l_1+2^{-r_j}m_{r_j}$ one obtains
\begin{eqnarray*}
\lefteqn{\frac{\#\{s:d(f^s(x_{\hat \alpha}),f^s(x_{\hat \beta}))<d_0, 0\le s<\nu_j\}}{\nu_j}}\\
&&{}\le\frac{\nu_j-\#\{s:(e_s \ldots e_{l_1+s-1}) \in E_{m_{r_j}}, s\ge 0\}}{\nu_j}\\
&&{}=\frac{\nu_j-(m_{r_j}l_1-l_1+1)}{\nu_j}\\
&&{}=\frac{2^{-r_j}m_{r_j}+l_1-1}{m_{r_j}l_1+2^{-r_j} m_{r_j}}
\rightarrow 0(j \rightarrow \infty).
\end{eqnarray*}
From the definition of inferior limit of a sequence, we have
\begin{equation}\label{frm-ttldfzer}
\liminf_{n\to\infty}\frac{\#\{s:d(f^s(x_{\hat \alpha}),f^s(x_{\hat \beta}))<d_0,
0\le s<n\}}{n}=0.
\end{equation}
From (\ref{frm-ttudfone}) and (\ref{frm-ttldfzer}) we can see that any two different elements in~$D_x$ form a distributional chaotic pair and therefore the map~$f$ is distributionally chaotic of type 1.
\end{proof}
\begin{remark}
Consider the map $f$ of Example~\ref{exm-fgdcs}. For the periodic sequence~ $\alpha=(121212\ldots)$, $V_\alpha$ is a singleton. Thus the map $f$ is also distributionally chaotic of type 1.
\end{remark}
\begin{remark}
In Theorem 2 of \cite{kulopro} was proved that a strictly $A$-coupled-expanding map $f$ satisfying two additional assumptions is topologically conjugate to a subshift of finite type in a compact $f$-invariant subset. Moreover, using this facts, in the Remark 3.3 of \cite{kulopro} they said that if assumptions of the Theorem 2 are satisfied, then the $f$ is distributionally chaotic. However, Example~\ref{exm-fgdcs} does not satisfy the assumptions of the Theorem 2 from \cite{kulopro}.
\end{remark}


\vskip0.5cm\noindent
{\bf References}
\small\vskip0.4cm

\begin{thebibliography}{77}

\bibitem{balsmistef}
F.~Balibrea, J.~Sm\'{\i}tal, M.~\v{S}tef\'ankov\'a, \textit{The three versions of
distributional chaos}, Chaos Solitons Fractals~\textbf{23} (2005), 1581--1583.

\bibitem{blockcop}
L.~Block,  W.~Coppel, \textit{Dynamics on One Dimension}, Lecture Notes in  Mathematics Vol.~1513, Springer, Berlin, 1992.

\bibitem{fortivnc}
G.~L.~Forti, \textit{Various notions of chaos for discrete dynamical systems, a brief
survey}, Aequationes Math.~\textbf{70} (2005), 1--13.

\bibitem{kulopro}
M.~Kulczycki, P.~Oprocha, \textit{Coupled-expanding maps and matrix shifts}, Internat.\ J.\ Bifur.\ Chaos.~\textbf{23} (2013), (6 pages).

\bibitem{liuwangchudcdc}
H.~Liu, L.~Wang, Z.~Chu, \textit{Devaney's Chaos implies distributional chaos in a sequence}, Nonlinear Anal.~\textbf{71} (2009) 6144--6147.

\bibitem{misirewhsh}
M.~Misiurewicz, \textit{Horseshoes for mappings of the interval},
Bull.\ Acad.\ Polon.\ Sci.\ S\'{e}r.\ Sci.\ Math.~\textbf{27} (1979), 167--169.

\bibitem{misirehsbr}
M.~Misiurewicz, \textit{Horseshoes for continuous mappings of an interval}, In:
\textit{Dynamical systems \textup{(}Bressanone, \textup{1978)}}, pp.~125--135, Liguori,
Naples, 1980.

\bibitem{oprochaddc}
P.~Oprocha, \textit{Relations between distributional and Devaney chaos}, Chaos~\textbf{16} (2006), 033112, 5~pp.

\bibitem{oprochadcsc}
P.~Oprocha, \textit{Distributional chaos via semiconjugacy}, Nonlinearity~\textbf{20}
(2007), 2661--2679.

\bibitem{oprochadcr}
P.~Oprocha, \textit{Distributional chaos revisited}, Trans.\ Amer.\ Math.\ Soc.~\textbf{361} (2009), 4901--4925.

\bibitem{rijuentacem}
Ch.~Ri, H.~Ju, \textit{Entropy for $A$-coupled expanding map and chaos}, Preprint.

\bibitem{robinsondynsys}
C.~Robinson, \textit{Dynamical system {\textendash} Stability, Symbolic Dynamics and Chaos}, CRC Press, Boca Raton, 1999.

\bibitem{schweizsmitdc}
B.~Schweizer, J.~Sm\'{\i}tal, \textit{Measures of chaos and a spectral decomposition of
dynamical systems on the interval}, Trans.\ Amer.\ Math.\ Soc.~\textbf{344} (1994),
737--754.

\bibitem{shichenccms}
Y.~Shi, G.~Chen, \textit{Chaos of discrete dynamical systems in complete metric spaces}, Chaos Solitons Fractals~\textbf{22} (2004), 555--571.

\bibitem{shichencep}
Y.~Shi, G.~Chen, \textit{Some new criteria of chaos induced by coupled-expanding maps},
In Proc.~1\textsuperscript{st}~IFAC Conf.\ Analysis and Control of Chaotic systems, Reims, France, pp.~157--162.

\bibitem{shijuchen}
Y.~Shi, H.~Ju, G.~Chen, \textit{Coupled-expanding maps and
one-sided symbolic dynamical  systems}, Chaos Solitons Fractals~\textbf{39} (2009),
2138--2149.

\bibitem{shiyuturb}
Y.~Shi, P.~Yu, \textit{Study on chaos induced by turbulent maps in noncompact sets},
Chaos Solitons Fractals~\textbf{28} (2006), 1165--1180.

\bibitem{shiyusbrep}
Y.~Shi, P.~Yu, \textit{Chaos induced by regular snap-back repellers}, J.\ Math.\ Anal.\
Appl.~\textbf{337} (2008), 1480--1494.

\bibitem{smitalvnc}
J.~Sm\'{\i}tal, \textit{Various notions of chaos, recent results, open problems}, Report
on the Summer Symposium in Real Analysis~XXVI, Real Anal.\ Exchange~2002,
26\textsuperscript{th}~Summer Symposium Conference, Suppl., pp.~81--85

\bibitem{wangxiang}
H.~Wang, J.~Xiang, \textit{Chaos for subshifts of finite type}, Acta Math.\ Sinica~\textbf{21} (2005), 1407--1414.

\bibitem{wanghuanghuan}
L.~Wang, G.~Huang, S.~Huan, \textit{Distributional chaos in a sequence}, Nonlinear
Anal.~\textbf{67} (2007), 2131--2136.

\bibitem{shizhangcem}
X.~Zhang, Y.~Shi, \textit{Coupled-expanding maps for ireducible transition matrices},
Internat.\ J.\ Bifur.\ Chaos.~\textbf{20} (2010), 3769--3783.

\bibitem{zhangshichencem}
X.~Zhang , Y.~Shi, G.~Chen, \textit{Some properties of coupled-expanding maps in compact
sets}, Preprint.

\bibitem{zhangshis}
ZS.~Zhang, \textit{On the shift invariant sets of self-maps}, Acta Math.\
Sin.~\textbf{27} (1984), 564--576.

\end{thebibliography}
\end{document}